\documentclass[11pt, reqno]{amsart}
\usepackage{amssymb}
\usepackage{verbatim}
\usepackage[vcentermath]{youngtab}
\usepackage{float}

\usepackage{times}
\usepackage[T1]{fontenc}
\usepackage{mathrsfs}
\usepackage{epsfig}
\usepackage{color}
\usepackage{array}
\usepackage{enumitem}

\newcolumntype{L}{>{$}l<{$}}
\newcolumntype{R}{>{$}r<{$}}

\newtheorem{thm}{Theorem}[section]

\newtheorem{cor}[thm]{Corollary}

\newtheorem{question}[thm]{Question}

\theoremstyle{definition}

\newtheorem{example}[thm]{Example}

\numberwithin{equation}{section}

\def\ZZ{\mathbb{Z}}
\def\QQ{\mathbb{Q}}

\def\AA{\mathbb{A}}

\def\FF{\mathbb{F}}

\def\CC{\mathbb{C}}
\def\RR{\mathbb{R}}

\def\Conf{\mathrm{Conf}}

\newcommand\PConf{\mathrm{PConf}}

\def\sgn{\mathrm{sgn}}
\def\multi#1#2{\ensuremath{\left(\kern-.3em\left(\genfrac{}{}{0pt}{}{#1}{#2}\right)\kern-.3em\right)}}

\newcommand{\Sym}{\mathrm{Sym}}
\newcommand{\poly}{\mathrm{Poly}}
\newcommand{\sfr}{\mathrm{sf}}
\newcommand{\sfpoly}{\mathrm{Poly}^{\mathrm{sf}}}

\newcommand{\ch}{\mathrm{ch}}
\newcommand{\lie}{\mathrm{Lie}}
\newcommand{\rank}{\mathrm{rk}}

\newcommand{\Sp}{\mathcal{S}}
\newcommand{\tr}{\mathbf{1}}
\newcommand{\Sgn}{\mathbf{Sgn}}
\newcommand{\Std}{\mathbf{Std}}

\begin{document}

\title[Polynomial factorization statistics]{Polynomial factorization statistics and point configurations in $\RR^3$}

\author{Trevor Hyde}
\address{Dept. of Mathematics\\
University of Michigan \\
Ann Arbor, MI 48109-1043\\
}
\email{tghyde@umich.edu}

\date{December 11th, 2017}

\maketitle

\begin{abstract}
    We use generating functions to relate the expected values of polynomial factorization statistics over $\FF_q$ to the cohomology of ordered configurations in $\RR^3$ as a representation of the symmetric group. Our methods lead to a new proof of the twisted Grothendieck-Lefschetz formula for squarefree polynomial factorization statistics of Church, Ellenberg, and Farb.
\end{abstract}

\section{Introduction}
Arithmetic statistics is the study of rings like $\ZZ$ and $\FF_q[x]$ from a statistical point of view. For example, one may ask: what is the probability that a random integer $m$ in the interval $[1,n]$ is prime? The Prime Number Theorem tells us that
\begin{equation}
\label{eqn prime prob}
    \mathrm{Prob}(m \in [1,n]\text{ is prime}) \approx \frac{1}{\log(n)}
\end{equation}
for sufficiently large $n$. A similar question may be asked about $\FF_q[x]$: what is the probability that a random monic degree $d$ polynomial $f(x) \in \FF_q[x]$ is irreducible? It may be shown that
\begin{equation}
\label{eqn 1}
    \mathrm{Prob}(f(x)\text{ monic degree $d$ is irreducible}) \approx \frac{1}{d}
\end{equation}
for sufficiently large $q$. Note that $\frac{1}{d} = \frac{1}{\log_q (q^d)}$ and there are $q^d$ monic degree $d$ polynomials in $\FF_q[x]$; hence \eqref{eqn 1} parallels \eqref{eqn prime prob}.

We can do much better than this approximate answer to the irreducibility question in $\FF_q[x]$. The number of irreducible monic degree $d$ polynomials in $\FF_q[x]$ is given by \emph{$d$th necklace polynomial} $M_d(q)$,
\[
    \#\{f(x) \text{ irreducible monic degree $d$}\} = M_d(q) := \frac{1}{d}\sum_{e\mid d}\mu(d/e)q^e,
\]
where $\mu$ is the M\"{o}bius function. Hence we have the exact expression,
\begin{equation}
\label{eqn irred prob} 
    \mathrm{Prob}(f(x)\text{ monic degree $d$ is irreducible}) = \frac{1}{d}\sum_{e\mid d}\frac{\mu(d/e)}{q^{d-e}}.
\end{equation}

There is a tendency to focus on the leading terms in asymptotic formulas like \eqref{eqn 1} and to ignore the lower degree terms as noise to be silenced. But if we look at the \emph{exact} answers on the $\FF_q[x]$ side of the analogy we see a different picture. If $d = 6$, then \eqref{eqn irred prob} specializes to
\begin{equation}
\label{eqn first example}
    \mathrm{Prob}(f(x)\text{ monic degree $6$ is irreducible}) = \frac{1}{6}\bigg(1 - \frac{1}{q^3} - \frac{1}{q^4} + \frac{1}{q^5}\bigg).
\end{equation}
Each term in \eqref{eqn first example} corresponds to an intermediate field of the degree 6 extension $\FF_{q^6}/\FF_q$ and the coefficients encode how these fields fit together---far from noise! This is an instance of a more general theme: the exact expressions for arithmetic statistical questions in $\FF_q[x]$ reflect hidden structure which is not apparent from the leading terms alone. In other words, there are no error terms; each term has an interpretation and together they tell a complete story.

In this paper we consider another family of arithmetic statistics questions for which the exact answers exhibit rich structure. We begin with an example based on \cite[Pg. 6]{CEF}. Define the \emph{quadratic excess} $Q(f)$ of a polynomial $f(x) \in \FF_q[x]$ to be
\begin{align*}
    Q(f) = \hspace{.5em}&\#\{\text{reducible quadratic factors of $f(x)$}\}\\
    &- \#\{\text{irreducible quadratic factors of $f(x)$}\},
\end{align*}
where both counts are considered with multiplicity. Note that $Q(f)$ depends only on the number of linear and irreducible quadratic factors of $f(x)$. For instance, if $g(x) = x^2(x+1)(x^2 +1)^4 \in \FF_3[x]$, then $g(x)$ has 3 linear factors and 4 irreducible quadratic factor, hence
\[
    Q(g) = \binom{3}{2} - \binom{4}{1} = -1.
\]
We write $E_d(Q)$ for the expected value of $Q$ on the set $\poly_d(\FF_q)$ of monic degree $d$ polynomials in $\FF_q[x]$. That is,
\[
    E_d(Q) := \frac{1}{q^d}\sum_{f\in \poly_d(\FF_q)}Q(f).
\]
The table below gives the expected value $E_d(Q)$ for small values of $d$.
\begin{center}
\begin{tabular}{c|l}
    $d$ & $E_d(Q)$\\
\hline
    $3$ & $\tfrac{2}{q} + \tfrac{1}{q^2}$ \\
    $4$ & $\tfrac{2}{q} + \tfrac{2}{q^2} + \tfrac{2}{q^3}$\\
    $5$ & $\tfrac{2}{q} + \tfrac{2}{q^2} + \tfrac{4}{q^3} + \tfrac{2}{q^4}$\\
    $6$ & $\tfrac{2}{q} + \tfrac{2}{q^2} + \tfrac{4}{q^3} + \tfrac{4}{q^4} + \tfrac{3}{q^5}$\\
    $10$ & $\tfrac{2}{q} + \tfrac{2}{q^2} + \tfrac{4}{q^3} + \tfrac{4}{q^4} + \tfrac{6}{q^5} + \tfrac{6}{q^6} + \tfrac{8}{q^7} + \tfrac{8}{q^8} + \tfrac{5}{q^9}$\\
\end{tabular}
\end{center}

We note a few remarkable features of these expected values. For each $d$, $E_d(Q)$ is a polynomial in $\frac{1}{q}$ of degree $d - 1$ with \emph{positive integer coefficients}; one should expect the coefficients to be rational numbers, but both the positivity and integrality are not a priori evident. Evaluating the polynomial $E_d(Q)$ at $q = 1$ gives the binomial coefficient $\binom{d}{2}$. The coefficients of $E_d(Q)$ appear to stabilize as $d$ increases with a clear pattern emerging already for $d = 10$, suggesting that the expected values $E_d(Q)$ converge coefficientwise as $d\rightarrow\infty$.

These phenomena are not isolated to the function $Q$ but rather appear for a family of functions on $\poly_d(\FF_q)$ which we call \emph{factorization statistics}. Our main result identifies the structure underlying the expected values of factorization statistics on $\poly_d(\FF_q)$.

\begin{thm}
\label{thm twist intro}
Let $P$ be a function defined on the set $\poly_d(\FF_q)$ of monic degree $d$ polynomials in $\FF_q[x]$ such that $P(f)$ only depends on the degrees of the irreducible factors of $f$. $P$ may also be viewed as a function defined on partitions of $d$, or equivalently as a class function of the symmetric group $S_d$.

Let $\psi_d^k$ be the character of the $S_d$-representation $H^{2k}(\PConf_d(\RR^3),\QQ)$ where $\PConf_d(\RR^3)$ is the ordered configuration space of $d$ distinct points in $\RR^3$ (see Section \ref{section twisted}.) Then the expected value $E_d(P)$ of $P$ on $\poly_d(\FF_q)$ is given by
\[
    E_d(P) := \frac{1}{q^d}\sum_{f\in \poly_d(\FF_q)}P(f) = \sum_{k=0}^{d-1}\frac{\langle P, \psi_d^k\rangle}{q^k},
\]
where $\langle P, \psi_d^k\rangle := \frac{1}{d!}\sum_{\sigma\in S_d} P(\sigma)\psi_d^k(\sigma)$ is the standard inner product of class functions of the symmetric group $S_d$.
\end{thm}

Factorization statistics are functions $P$ satisfying the hypotheses of Theorem \ref{thm twist intro}. Theorem \ref{thm twist intro} relates the expected value of factorization statistics to the representation theory of the symmetric group and the topology of configuration spaces. We return to the quadratic excess example and explore this interplay between representation theory, topology, and the combinatorics of finite fields through examples in Section \ref{section example}.

Recall that $\FF_q[x]$ has unique factorization, hence any monic degree $d$ polynomial $f(x) \in \FF_q[x]$ may be uniquely factored over $\FF_q$ into a product of irreducible polynomials. The degrees of these irreducible factors form a partition $\lambda$ of the degree $d$ which we call the \emph{factorization type} of $f(x)$. Given a partition $\lambda \vdash d$ let $\nu(\lambda)$ denote the probability of a random monic degree $d$ polynomial in $\FF_q[x]$ having factorization type $\lambda$. We call $\nu$ the \emph{splitting measure}. The key to proving Theorem \ref{thm twist intro} is to first give a representation theoretic interpretation of the splitting measure.

\begin{thm}
Let $\psi_d^k$ be the character of the $S_d$-representation $H^{2k}(\PConf_d(\RR^3),\QQ)$ where $\PConf_d(\RR^3)$ is the ordered configuration space of $d$ distinct points in $\RR^3$ (see Section \ref{section twisted}.) Then for all $d\geq 1$ and partitions $\lambda \vdash d$ we have
\[
    \nu(\lambda) = \frac{1}{z_\lambda}\sum_{k=0}^{d-1}\frac{\psi_d^k(\lambda)}{q^k},
\]
where $z_\lambda := \prod_{j\geq 1} j^{m_j} m_j!$ when $\lambda = (1^{m_1}2^{m_2}\cdots)$, and $\psi_d^k(\lambda)$ is the value of the character $\psi_d^k$ on any element of the symmetric group $S_d$ with cycle type $\lambda$.
\end{thm}

\subsection{Related work}

Church, Ellenberg, and Farb \cite{CEF} connect the first moments of factorization statistics on the set $\poly_d^\sfr(\FF_q)$ of \emph{squarefree} monic degree $d$ polynomials to the symmetric group representations carried by the cohomology of configuration space through their twisted Grothendieck-Lefschetz formula for $\poly_d^\sfr(\FF_q)$.

\begin{thm}[{\cite[Prop. 4.1]{CEF}}]
\label{thm CEF twist}
Let $\phi_d^k$ be the character of the $S_d$-representation $H^k(\PConf_d(\CC),\QQ)$ where $\PConf_d(\CC)$ is the ordered configuration space of $d$ distinct points in $\CC$. Let $\poly_d^\sfr(\FF_q)$ denote the set of squarefree monic degree $d$ polynomials in $\FF_q[x]$. Then for any factorization statistic $P$,
\begin{equation}
\label{eqn CEF formula}
    \sum_{f\in\poly_d^\sfr(\FF_q)} P(f) = \sum_{k=0}^{d-1}(-1)^k \langle P, \phi_d^k\rangle q^{d-k},
\end{equation}
where $\langle P, \phi_d^k\rangle := \frac{1}{d!}\sum_{\sigma\in S_d}P(\sigma)\phi_d^k(\sigma)$ is the standard inner product of class functions of the symmetric group $S_d$.
\end{thm}

The first moment formula \eqref{eqn CEF formula} is derived from the Grothendieck-Lefschetz trace formula for \'{e}tale cohomology with ``twisted coefficients'' from which the name is borrowed. The author and Lagarias \cite{Hyde Lagarias} use Theorem \ref{thm CEF twist} to establish a representation theoretic interpretation of the \emph{squarefree splitting measure} $\nu^\sfr$, where $\nu^\sfr(\lambda)$ is the probability of a random squarefree polynomial having factorization type $\lambda$.

\begin{thm}[{\cite[Thm. 1.2]{Hyde Lagarias}}]
\label{thm sf rep interp}
Let $\chi_d^k$ be the character of the $S_d$-representation $H^k(\PConf_d(\CC)/\CC^\times,\QQ)$ (see Section \ref{sec config}.) Then for all $d\geq 2$ and partitions $\lambda\vdash d$ we have
\[
    \nu^\sfr(\lambda) = \frac{1}{z_\lambda}\sum_{k=0}^{d-2}\frac{(-1)^k \chi_d^k(\lambda)}{q^k},
\]
where $z_\lambda = \prod_{j\geq 1} j^{m_j} m_j!$ when $\lambda = (1^{m_1}2^{m_2}\cdots)$, and $\chi_d^k(\lambda)$ is the value of the character $\chi_d^k$ on any element of the symmetric group $S_d$ with cycle type $\lambda$.
\end{thm}

We give a new proof of Theorem \ref{thm sf rep interp} using generating functions and derive Theorem \ref{thm CEF twist} as a consequence. Our proofs of Theorem \ref{thm twist intro} and Theorem \ref{thm CEF twist} do not use algebraic geometry or the Grothendieck-Lefschetz trace formula, but we keep the name ``twisted Grothendieck-Lefschetz formula'' to emphasize the parallel our Theorem \ref{thm twist intro} and the result of Church, Ellenberg, and Farb.

The use of generating functions in the study of factorization statistics is not new. Church, Ellenberg, and Farb \cite{CEF} use $L$-functions to compute the stable limits of expected values of squarefree factorization statistics. Fulman \cite{Fulman} uses cycle index series to derive the asymptotic formulas for first moments of squarefree factorization statistics given in \cite{CEF} without using representation theory or cohomology. Chen \cite{Chen 1, Chen 2} further develops these methods in the more general setting of an arbitrary affine or projective variety $V$ defined over $\FF_q$. Our main innovations are the application of generating functions to factorization statistics on the set of \emph{all} polynomials (instead of the subset of squarefree polynomials) and the derivation of the twisted Grothendieck-Lefschetz formulas by generating function methods.

There have been other generalizations of Theorem \ref{thm CEF twist} from squarefree polynomials to all polynomials. Gadish \cite[Sec. 1.3]{Gadish} and Hast, Matei \cite{Hast Matei} both study expected values of functions defined on the set of all polynomials; their functions depend on both the degree of the irreducible factors and their multiplicities. We call these \emph{weighed factorization statistics}. Gadish \cite[Cor. 1.4]{Gadish} shows that the expected value of a weighted factorization statistic $P$ on $\poly_d(\FF_q)$ matches the expected value of $P$ on $S_d$ viewed as a class function. Stated geometrically, the expected values of weighted factorization statistics on degree $d$ polynomials correspond to the cohomology of $\RR^d$ as an $S_d$-representation, while the expected values of our factorization statistics correspond to the cohomology of $\PConf_d(\RR^3)$ as an $S_d$-representation.

\subsection{Further questions}
Church, Ellenberg, and Farb's \'{e}tale cohomology approach to Theorem \ref{thm CEF twist} shows a clear geometric connection between squarefree factorization statistics and the cohomology of ordered configurations in $\CC$. To summarize, we start with a map of schemes
\begin{equation}
\label{eqn map}
    \PConf_d(\AA^1) \longrightarrow \Conf_d(\AA^1),
\end{equation}
which sends an ordered configuration of $d$ points to its unordered counterpart. The symmetric group $S_d$ acts on $\PConf_d(\AA^1)$ by permuting points in the ordered configuration, and the map in \eqref{eqn map} is the quotient by this action. The $\FF_q$-points of $\Conf_d(\AA^1)$ are in natural correspondence with squarefree polynomials of degree $d$, and the $\CC$-points of $\PConf_d(\AA^1)$ give us the manifold $\PConf_d(\CC)$. The Grothendieck-Lefschetz trace formula connects point counts over finite fields with the \'{e}tale cohomology of the scheme and general comparison theorems between cohomology theories relate this to the singular cohomology of the manifold $\PConf_d(\CC)$.

The map \eqref{eqn map} is unramified, simplifying the application of the Grothendieck-Lefschetz trace formula. The corresponding map of schemes in the case of all polynomials is
\[
    (\AA^1)^d \longrightarrow \Sym_d(\AA^1),
\]
which is highly ramified. Gadish \cite{Gadish} adapts the \'{e}tale cohomological approach to handle ramified covers. This geometrically natural extension leads Gadish to a twisted Grothendieck-Lefschetz formula for weighted factorization statistics \cite[Thm. A (1.2)]{Gadish}.

Our factorization statistics extend those on $\sfpoly_d(\FF_q)$ in a way that is combinatorially natural but is at odds with the apparatus of algebraic geometry. This results in a surprising connection to the cohomology of ordered configurations in $\RR^3$ for which we have no geometric intuition.

\begin{question}
Is there a geometric explanation for the connection between factorization statistics on $\poly_d(\FF_q)$ and the cohomology of $\PConf_d(\RR^3)$?
\end{question}

Church, Ellenberg, and Farb deduce their twisted Grothendieck-Lefschetz formula from a more general result relating factorization statistics on quotients of complements of hyperplane arrangements to the \'{e}tale cohomology of said complements. Note that $\PConf_d(\CC)$ may be interpreted as the complement of the \emph{braid arrangement}, consisting of the hyperplanes $z_i = z_j$ for all $i\neq j$. Given a collection of linear forms $L$ defined over $\ZZ$ in $d$ variables which is stable under the natural action of $S_d$, let $A_d(L)$ be the complement of the hyperplane arrangement determined by the vanishing sets of the linear forms. Let $B_d(L)$ denote the scheme-theoretic quotient of $A_d(L)$ by the action of $S_d$.

\begin{thm}[{\cite[Thm. 3.7]{CEF}}]
\label{thm CEF general}
Let $P$ be a factorization statistic. If $\ell$ is a prime coprime to $q$ and $\tau_d^k$ is the character of $H_{\text{\'{e}t}}^k(A_d(L), \QQ_{\ell})$, then
\[
    \sum_{f \in B(L)_d(\FF_q)} P(f) = \sum_{k=0}^d (-1)^k \langle P, \tau_d^k\rangle q^{d-k}.
\]
\end{thm}

\noindent Given that our generating function method provides a new proof of the special case Theorem \ref{thm CEF twist}, we ask:

\begin{question}
Can our generating function methods be adapted to give a new proof of Theorem \ref{thm CEF general}?
\end{question}

\noindent The key is to find explicit product formulas for the cycle index series of the family of representations given by the \'{e}tale cohomology analogous to those which we use in the proof of Theorem \ref{thm rep interp}. Such formulas may be known, but not to us.

Finally we pose the broad question of how to adapt these results on factorization statistics to the integers $\ZZ$. It is not clear what the correct analog should be and what form the results should take. 


\subsection*{Acknowledgements}
We would like to thank Jeff Lagarias for asking the question whether there may be interesting structure in the splitting measure of $\poly_d(\FF_q)$.

We also thank Weiyan Chen, Will Sawin, Phil Tosteson, and Michael Zieve for helpful conversations, references, and feedback.


\section{Twisted Grothendieck-Lefschetz formulas}
\label{section twisted}

Let $q$ be a prime power and $d\geq 1$ an integer. Let $\poly_d(\FF_q)$ be the set of monic degree $d$ polynomials in $\FF_q[x]$. The subset of squarefree polynomials is denoted $\sfpoly_d(\FF_q) \subseteq \poly_d(\FF_q)$. Every polynomial $f \in \poly_d(\FF_q)$ has a unique factorization into irreducible polynomials over $\FF_q$. The degrees of the irreducible factors of $f$ form a partition $[f]$ of the degree $d$ which we call the \emph{factorization type} of $f$. Recall that the number of degree $j$ irreducible polynomials in $\FF_q[x]$ is given by the \emph{necklace polynomial}
\[
    M_j(q) := \frac{1}{j} \sum_{i \mid j} \mu(i)q^{j/i}.
\]
If $\lambda$ is a partition, then $m_j = m_j(\lambda)$ is the number of size $j$ parts of $\lambda$. In other words, $\lambda = (1^{m_1} 2^{m_2} 3^{m_3}\cdots)$. Given a partition $\lambda \vdash d$ we define $\multi{\AA^1}{\lambda}$ and $\binom{\AA^1}{\lambda}$ to be the number of polynomials in $\poly_d(\FF_q)$ and $\sfpoly_d(\FF_q)$ respectively with factorization type $\lambda$. The $\AA^1$ in the notation reflects the correspondence between degree $d$ monic polynomials in $\FF_q[x]$ and points in $\Sym_d(\AA^1)(\FF_q)$. By unique factorization we have the formulas,
\begin{equation}
\label{eqn explicit count}
    \multi{\AA^1}{\lambda} := \prod_{j\geq 1}\multi{M_j(q)}{m_j} \hspace{3em}
    \binom{\AA^1}{\lambda} := \prod_{j\geq 1}\binom{M_j(q)}{m_j},
\end{equation}
where 
\[
    \multi{x}{m} := \frac{x(x + 1)(x + 2)\cdots (x + m - 1)}{m!} = \binom{x + m - 1}{m}.
\]
Note that $\multi{x}{m}$ counts the number of subsets of size $m$ chosen from an $x$ element set with repetition. The product expressions \eqref{eqn explicit count} show that $\multi{\AA^1}{\lambda}$ and $\binom{\AA^1}{\lambda}$ are polynomials in $q$ of degree $d$.

The total number of monic degree $d$ polynomials over $\FF_q$ is $|\poly_d(\FF_q)| = q^d$, while the total number of squarefree polynomials is $|\sfpoly_d(\FF_q)| = q^d - q^{d-1}$ for $d\geq 2$ (see \cite[Prop. 2.3]{Rosen}.) We define the \emph{splitting measure} $\nu(\lambda)$ to be the probability of an element $f \in \poly_d(\FF_q)$ having factorization type $\lambda$, and similarly define the \emph{squarefree splitting measure} $\nu^{\sfr}(\lambda)$ for $f \in \sfpoly_d(\FF_q)$. More explicitly,
\[
    \nu(\lambda) := \frac{1}{|\poly_d(\FF_q)|}\multi{\AA^1}{\lambda} \hspace{3em}
    \nu^\sfr(\lambda) := \frac{1}{|\poly_d^\sfr(\FF_q)|}\binom{\AA^1}{\lambda}.
\]

Both splitting measures are rational functions in $q$ for each partition $\lambda$, and furthermore both are polynomials in $1/q$ (this is clear for $\nu(\lambda)$ and is shown for $\nu^\sfr(\lambda)$ in \cite[Prop. 2.4]{Hyde Lagarias}.) Recall that the partitions $\lambda \vdash d$ parametrize the conjugacy classes of the symmetric group $S_d$. Thus the splitting measures may be viewed as polynomial-valued class functions on $S_d$. Our first result gives an interpretation of the coefficients of the splitting measures in terms of the representation theory of the symmetric group.

Let us review some terminology and notation. If $\chi$ is a character of the symmetric group $S_d$ and $\lambda$ is a partition of $d$, we write $\chi(\lambda)$ for the value of $\chi$ on any element $\sigma \in S_d$ of cycle type $\lambda$. This is well-defined since characters are constant on conjugacy classes. Let $z_\lambda$ be the number of permutations in $S_d$ commuting with an element $\sigma \in S_d$ of cycle type $\lambda$, then
\[
    z_\lambda := \prod_{j\geq 1}j^{m_j}m_j!.
\]
The \emph{rank} of a partition $\lambda \vdash d$ is $\rank(\lambda) := \sum_{j\geq 1} m_j - 1 = d - \ell(\lambda)$, where $\ell(\lambda)$ is the number of parts in $\lambda$.

We now introduce the $S_d$-representations whose characters are shown in Theorem \ref{thm rep interp} to determine the coefficients of the splitting measures.

\subsection{Higher Lie representations} Given a positive integer $j$, let $\zeta_j$ be a faithful one-dimensional complex representation of the cyclic group $C_j$. Viewing $C_j$ as a subgroup of the symmetric group $S_j$ generated by a $j$-cycle, the \emph{$j$th Lie representation} $\lie(j)$ is defined as
\[
    \lie(j) := \mathrm{Ind}_{C_j}^{S_j} \zeta_j.
\]
For a partition $\lambda \vdash d$, the \emph{higher Lie representation} $\lie_\lambda$ is defined as
\[
    \lie_\lambda := \mathrm{Ind}_{Z_\lambda}^{S_d} \bigotimes_{j\geq 1}\lie(j)^{\otimes m_j(\lambda)},
\]
where $Z_\lambda$ is the centralizer of a permutation with cycle type $\lambda$. Finally, for $0\leq k \leq d$ let $\lie_d^k$ be the $S_d$-representation
\[
    \lie_d^k := \bigoplus_{\rank(\lambda) = k} \lie_\lambda.
\]

\subsection{Configuration space}
\label{sec config}
Given a topological space $X$, let $\PConf_d(X)$ be the space of ordered configurations of $d$ distinct points in $X$,
\[
    \PConf_d(X) := \{(x_1,x_2,\ldots,x_d)\in X^d : x_i \neq x_j \text{ when } i\neq j\}.
\]
The symmetric group $S_d$ acts freely on $\PConf_d(X)$ by permuting the coordinates. Thus the singular cohomology $H^k(\PConf_d(X),\QQ)$ is, by functoriality, an $S_d$-representation for all $k\geq 0$. Sundaram and Welker \cite{Sundaram Welker} show for $k\geq 0$ that for every odd $n \geq 3$
\[
    H^{(n-1)k}(\PConf_d(\RR^n),\QQ) \cong \lie_d^k,
\]
as $S_d$-representations (see \cite[Sec. 2.3]{Hersh Reiner} for a discussion of this result in language closer to our presentation.) For the sake of concreteness we specialize to the case $n = 3$,
\[
    H^{2k}(\PConf_d(\RR^3),\QQ) \cong \lie_d^k.
\]

If $X = \CC$, then the unit group $\CC^\times$ acts on $\PConf_d(\CC)$ by simultaneously scaling all coordinates; this action commutes with $S_d$, hence there is a well-defined $S_d$-action on the quotient $\PConf_d(\CC)/\CC^\times$. Thus $H^k(\PConf_d(\CC)/\CC^\times,\QQ)$ is an $S_d$-representation for all $k\geq 0$.

We now state our first result.

\begin{thm}
\label{thm rep interp}
Let $\psi_d^k$ and $\chi_d^k$ be the characters of the $S_d$-representations $\lie_d^k \cong H^{2k}(\PConf_d(\RR^3),\QQ)$ and $H^k(\PConf_d(\CC)/\CC^\times,\QQ)$ respectively.
\begin{enumerate}
    \item For $d\geq 1$ and every partition $\lambda \vdash d$,
        \[
            \nu(\lambda) = \frac{1}{z_\lambda}\sum_{k=0}^{d-1} \frac{\psi_d^k(\lambda)}{q^k}.
        \]
    \item For $d\geq 2$ and every partition $\lambda \vdash d$,
        \[
            \nu^\sfr(\lambda) = \frac{1}{z_\lambda}\sum_{k=0}^{d-2} \frac{(-1)^k \chi_d^k(\lambda)}{q^k}.
        \]
\end{enumerate}
\end{thm}

This representation theoretic interpretation of the squarefree splitting measure was first shown in \cite[Thm. 5.1]{Hyde Lagarias} using the twisted Grothendieck-Lefschetz formula for squarefree factorization statistics of Church, Ellenberg, and Farb \cite[Prop. 4.1]{CEF}. We prove Theorem \ref{thm rep interp} using generating functions, leading to a new proof of the twisted Grothendieck-Lefschetz formula for squarefree factorization statistics in Theorem \ref{thm sf twisted gl}. The representation theoretic interpretation of the splitting measure $\nu(\lambda)$ appears to be new.

\begin{proof}
(1) For each integer $j\geq 1$ let $a_j$ be a formal variable. If $\lambda = (1^{m_1}2^{m_2}\cdots)$ is a partition, let $a^\lambda := \prod_{j\geq 1}a_j^{m_j}$. Hersh and Reiner \cite[Thm. 2.17]{Hersh Reiner} state the following identity of formal power series
\begin{equation}
\label{eqn first product identity}
    \sum_{d\geq 0}\sum_{\lambda\vdash d} \frac{1}{z_\lambda}\sum_{k=0}^{d-1} \psi_d^k(\lambda) q^{d-k} a^\lambda t^d = \prod_{j\geq 1}\left(\frac{1}{1 - a_j t^j}\right)^{M_j(q)},
\end{equation}
where $M_j(q) = \frac{1}{j}\sum_{i\mid j}\mu(i)q^{j/i}$ is the $j$th necklace polynomial and $\psi_d^k$ is the character of $\lie_d^k$ (see remarks following the proof for a discussion of the equivalence of \ref{eqn first product identity} and \cite[Thm. 2.17]{Hersh Reiner}.) Recall the formal identity,
\[
    \left(\frac{1}{1 - t}\right)^M = \sum_{m\geq 0} \multi{M}{m}t^m.
\]
Expanding the right hand side of \eqref{eqn first product identity} we have
\begin{align}
\begin{split}
    \prod_{j\geq 1}\left(\frac{1}{1 - a_jt^j}\right)^{M_j(q)} &= \prod_{j\geq 1}\sum_{m_j\geq 0}\multi{M_j(q)}{m_j} a_j^{m_j} t^{jm_j}\label{eqn second product identity}\\
    &= \sum_{d\geq 0} \sum_{\lambda \vdash d} \left(\prod_{j\geq 1}\multi{M_j(q)}{m_j(\lambda)}\right) a^\lambda t^d\\
    &= \sum_{d\geq 0} \sum_{\lambda \vdash d} \multi{\AA^1}{\lambda} a^\lambda t^d
\end{split}
\end{align}
Substitute $t = 1/q$ in \eqref{eqn first product identity} and \eqref{eqn second product identity} to find
\begin{align*}
    \sum_{d\geq 0}\sum_{\lambda\vdash d} \frac{1}{z_\lambda}\sum_{k=0}^{d-1} \frac{\psi_d^k(\lambda)}{q^k}a^\lambda &= \prod_{j\geq 1}\left(\frac{1}{1 - a_j/q^j}\right)^{M_j(q)}\\
    &= \sum_{d\geq 0} \sum_{\lambda \vdash d} \frac{1}{q^d}\multi{\AA^1}{\lambda} a^\lambda\\
    &= \sum_{d\geq 0} \sum_{\lambda \vdash d} \nu(\lambda) a^\lambda.
\end{align*}
Comparing coefficients of $a^\lambda$ we conclude that
\[
    \nu(\lambda) = \frac{1}{z_\lambda}\sum_{k=0}^{d-1} \frac{\psi_d^k(\lambda)}{q^k}.
\]

(2) To get the formula for $\nu^\sfr(\lambda)$ we start with another formal power series identity from \cite[Thm. 2.17]{Hersh Reiner}. Let $\phi_d^k$ be the character of the $S_d$-representation $H^k(\PConf_d(\CC),\QQ)$. Then
\begin{equation*}
    \sum_{d\geq 0}\sum_{\lambda\vdash d} \frac{1}{z_\lambda}\sum_{k=0}^{d-1} \phi_d^k(\lambda) q^{d-k} a^\lambda t^d = \prod_{j\geq 1}(1 + (-1)^ja_jt^j)^{M_j(-q)}.
\end{equation*}
The substitution $t \mapsto -t$ and $q \mapsto -q$ simplifies this to
\begin{equation}
\label{eqn sf first product identity}
    \sum_{d\geq 0}\sum_{\lambda\vdash d} \frac{1}{z_\lambda}\sum_{k=0}^{d-1} (-1)^k\phi_d^k(\lambda) q^{d-k} a^\lambda t^d = \prod_{j\geq 1}(1 + a_jt^j)^{M_j(q)}.
\end{equation}
By the binomial theorem, the right hand side of \eqref{eqn sf first product identity} expands as
\begin{align}
\label{eqn sf second product identity}
\begin{split}
    \prod_{j\geq 1}(1 + a_jt^j)^{M_j(q)} &= \prod_{j\geq 1}\sum_{m_j\geq 0}\binom{M_j(q)}{m_j}a_j^{m_j}t^{jm_j}\\
    &= \sum_{d\geq 0}\sum_{\lambda \vdash d}\left(\prod_{j\geq 1}\binom{M_j(q)}{m_j(\lambda)}\right) a^\lambda t^d\\
    &= \sum_{d\geq 0}\sum_{\lambda \vdash d}\binom{\AA^1}{\lambda} a^\lambda t^d.
\end{split}
\end{align}
Substituting $t = 1/q$ in \eqref{eqn sf first product identity} and \eqref{eqn sf second product identity} gives
\begin{align}
\begin{split}
    \sum_{d\geq 0}\sum_{\lambda\vdash d} \frac{1}{z_\lambda}\sum_{k=0}^{d-1} \frac{\phi_d^k(\lambda)}{q^k}a^\lambda &= \prod_{j\geq 1}(1 + a_j/q^j)^{M_j(q)}\label{eqn penultimate identity}\\
    &= \sum_{d\geq 0} \sum_{\lambda \vdash d} \frac{1}{q^d}\binom{\AA^1}{\lambda} a^\lambda.
\end{split}
\end{align}
Let $\chi_d^k$ be the character of the $S_d$-representation $H^k(\PConf_d(\CC)/\CC^\times,\QQ)$. Hyde and Lagarias \cite[Prop. 4.2, Thm. 4.3]{Hyde Lagarias} showed
\[
    H^k(\PConf_d(\CC),\QQ) \cong H^k(\PConf_d(\CC)/\CC^\times,\QQ) \oplus H^{k-1}(\PConf_d(\CC)/\CC^\times,\QQ),
\]
as $S_d$-representations from which it follows that $\phi_d^k = \chi_d^k + \chi_d^{k-1}$; note that $H^{-1}(\PConf_d(\CC)/\CC^\times,\QQ) = H^{d-1}(\PConf_d(\CC)/\CC^\times,\QQ) = 0$. Therefore
\begin{align}
\begin{split}
    \frac{1}{1-\frac{1}{q}}\sum_{k=0}^{d-1} \frac{(-1)^k \phi_d^k(\lambda)}{q^k} &= \frac{1}{1-\frac{1}{q}}\sum_{k=0}^{d- 1} \frac{(-1)^k\big(\chi_d^{k}(\lambda) + \chi_d^{k-1}(\lambda)\big)}{q^k}\label{eqn transition}\\
    &= \frac{1}{1-\frac{1}{q}}\sum_{k=0}^{d-2}\frac{(-1)^k\chi_d^k(\lambda)}{q^k} + \frac{(-1)^{k+1} \chi_d^k(\lambda)}{q^{k+1}}\\
    &= \sum_{k=0}^{d-2}\frac{(-1)^k\chi_d^k(\lambda)}{q^k}.
\end{split}
\end{align}
Multiplying \eqref{eqn penultimate identity} by $\frac{1}{1 - \frac{1}{q}}$ for $d\geq 2$ we have
\begin{align*}
    \sum_{d\geq 2}\sum_{\lambda\vdash d} \frac{1}{z_\lambda}\sum_{k=0}^{d-2} \frac{(-1)^k\chi_d^k(\lambda)}{q^k}a^\lambda &= \sum_{d\geq 2} \sum_{\lambda \vdash d} \frac{1}{q^d-q^{d-1}}\binom{\AA^1}{\lambda} a^\lambda\\
    &= \sum_{d\geq 2} \sum_{\lambda \vdash d} \nu^{\sfr}(\lambda) a^\lambda.
\end{align*}
Finally, comparing coefficients of $a^\lambda$ we conclude that for $d\geq 2$
\[
    \nu^{\sfr}(\lambda) = \frac{1}{z_\lambda}\sum_{k=0}^{d-2}\frac{(-1)^k\chi_d^k(\lambda)}{q^k}.
\]
\end{proof}

The generating functions used in the proof of Theorem \ref{thm rep interp} are stated in terms of symmetric functions in \cite{Hersh Reiner}. To convert between their notation and ours one can replace their power symmetric function $p_j$ with our formal variable $a_j$, and their Frobenius characteristic $\ch(V)$ of an $S_d$-representation $V$ with character $\chi_V$ by its \emph{cycle indicator} 
\[
    \ch(V) \longrightarrow \sum_{\lambda \vdash d} \frac{\chi_V(\lambda)}{z_\lambda} a^\lambda.
\]
Hersh and Reiner cite several sources for the origin of these generating functions. A derivation of the identity for the higher Lie characters may be found in \cite[Thm. 3.7]{Hanlon 1}, although the characters are not called by this name there. The generating function for the cohomology of configurations in $\CC$ is derived in \cite[Cor. 4.4]{Calderbank Hanlon Robinson} with notation similar to ours but stated in a way that does not explicitly connect it with configuration space.

\subsection{Factorization statistics and the twisted Grothendieck-Lefschetz formulas}
\label{subsec twist}
A \emph{factorization statistic} $P$ is a function defined on $\poly_d(\FF_q)$ such that $P(f)$ only depends on the factorization type of $f \in \poly_d(\FF_q)$. Equivalently, $P$ may be viewed as a function defined on the set of partitions of $d$ or as a class function of the symmetric group $S_d$. Any such function $P$ is called a factorization statistic when we want to think of $P$ as a function of polynomials.

\begin{example}
\label{ex 1}
We illustrate with some examples.
\begin{enumerate}[leftmargin=*]
    \item Consider the polynomials $g(x), h(x) \in \poly_5(\FF_3)$ with irreducible factorizations
    \[
        g(x) = x^2(x+1)(x^2 + 1) \hspace{3em} h(x) = (x+1)(x - 1)(x^3 - x + 1).
    \]
    The factorization type of $g(x)$ is the partition $[2,1,1,1] = (1^3\, 2^1)$ and the factorization type of $h(x)$ is $[3,1,1] = (1^2\, 3^1)$. Note that the factorization type does not detect the multiplicity of a specific factor so that $x^2$ and $x(x + 1)$ both have the same factorization type $[1,1]$.
    
    \item Let $R(f)$ be the number of $\FF_q$-roots of $f(x)\in \poly_d(\FF_q)$. Then $R(f)$ depends only on the number of linear factors of $f(x)$, hence is a factorization statistic. Referring to the two polynomials above we have $R(g) = 3$ and $R(h) = 2$.
    
    \item For $k\geq 1$, let $x_k(f)$ be the number of degree $k$ irreducible factors of $f \in \poly_d(\FF_q)$, then $x_k$ is a factorization statistic. As a function on partitions $x_k(\lambda) = m_k(\lambda)$ is the number of parts of $\lambda$ of size $k$. Note that $R = x_1$. The ring $\QQ[x_1, x_2, \ldots ]$ generated by the functions $x_k$ for $k\geq 1$ is called the ring of \emph{character polynomials}. We return to character polynomials in Section \ref{sec asymp stab} when we discuss asymptotic stability.
    
    \item Say a polynomial $f(x)$ has \emph{even type} if the factorization type of $f(x)$ is an even partition. In other words, say $\lambda = (1^{m_1} 2^{m_2} 3^{m_3}\cdots )$ is the factorization type of $f(x)$ and define $\sgn(\lambda)$ by
    \[
        \sgn(\lambda) = \prod_{j\geq 1} (-1)^{m_j(j - 1)},
    \]
    then $f(x)$ has even type if $\sgn(\lambda) = 1$. The indicator function $ET$ defined by
    \[
        ET(f) = \begin{cases} 1 & \text{$f(x)$ has even type}\\ 0 & \text{otherwise,}\end{cases}
    \]
    is a factorization statistic. Thus $ET(g) = 0$ and $ET(h) = 1$.
    
\end{enumerate}
\end{example}

We write $E_d(P)$ for the expected value of a factorization statistic $P$ on $\poly_d(\FF_q)$ and $E_d^{\sfr}(P)$ for the expected value of $P$ on $\sfpoly_d(\FF_q)$. More precisely,
\begin{align*}
    E_d(P) &:= \frac{1}{|\poly_d(\FF_q)|}\sum_{f\in \poly_d(\FF_q)} P(f)\\
    E_d^\sfr(P) &:= \frac{1}{|\poly_d^\sfr(\FF_q)|}\sum_{f\in \poly_d^\sfr(\FF_q)} P(f).
\end{align*}

Our second main result gives an explicit expression for the expected value $E_d(P)$ of a factorization statistic in terms of the ordered configuration space of $d$ distinct points in $\RR^3$. 

If $P$ and $Q$ are class functions on $S_d$, let $\langle P, Q\rangle$ denote their standard $S_d$-invariant inner product
\[
    \langle P, Q \rangle := \frac{1}{d!}\sum_{\sigma\in S_d} P(\sigma)Q(\sigma) = \sum_{\lambda\vdash d}\frac{P(\lambda)Q(\lambda)}{z_\lambda}.
\]

\begin{thm}[Twisted Grothendieck-Lefschetz for $\poly_d(\FF_q)$]
\label{thm twisted gl}
Suppose $P$ is a factorization statistic and $d\geq 1$. If $\psi_d^k$ is the character of the $S_d$-representation $\lie_d^k \cong H^{2k}(\PConf_d(\RR^3),\QQ)$, then
\[
    E_d(P) = \sum_{k=0}^{d-1} \frac{ \langle P, \psi_d^k \rangle}{q^k}.
\]
\end{thm}

\begin{proof}
Since factorization statistics depend only on the factorization type of a polynomial, we may rewrite the expected value in terms of the splitting measure,
\[
    E_d(P) = \frac{1}{|\poly_d(\FF_q)|}\sum_{f\in \poly_d(\FF_q)} P(f) = \sum_{\lambda \vdash d} P(\lambda)\nu(\lambda).
\]
Then Theorem \ref{thm rep interp} implies,
\begin{align*}
    E_d(P) &= \sum_{\lambda \vdash d} P(\lambda)\nu(\lambda)\\
    &= \sum_{\lambda\vdash d}\frac{1}{z_\lambda}\sum_{k=0}^{d-1} \frac{P(\lambda)\psi_d^k(\lambda)}{q^k}\\
    &= \sum_{k=0}^{d-1} \frac{1}{q^k}\left(\sum_{\lambda \vdash d} \frac{P(\lambda)\psi_d^k(\lambda)}{z_\lambda}\right)\\
    &= \sum_{k=0}^{d-1} \frac{\langle P, \psi_d^k\rangle}{q^k}.
\end{align*}
\end{proof}

Church, Ellenberg, and Farb relate the first moments of factorization statistics on squarefree polynomials to the ordered configuration space of $d$ distinct points in $\CC$ in \cite{CEF}. Let $\phi_d^k$ be the character of $H^k(\PConf_d(\CC),\QQ)$ as a representation of $S_d$. In \cite[Prop. 4.1]{CEF}, Church et al. show 
\begin{equation}
\label{eqn CEF version}
    \sum_{f\in\sfpoly_d(\FF_q)}P(f) = \sum_{k=0}^{d-1}(-1)^k \langle P, \phi_d^k \rangle q^{d-k}.
\end{equation}
Dividing by $|\sfpoly_d(\FF_q)| = q^d - q^{d-1}$ gives the expected value, but also changes the coefficients on the right hand side. The calculation \eqref{eqn transition} in the proof of Theorem \ref{thm rep interp} shows that the identity \eqref{eqn CEF version} is equivalent to Theorem \ref{thm sf twisted gl} below.

We give a new proof of \cite[Prop. 4.1]{CEF} using Theorem \ref{thm rep interp}.

\begin{thm}[Twisted Grothendieck-Lefschetz for $\sfpoly_d(\FF_q)$]
\label{thm sf twisted gl}
Suppose $P$ is a factorization statistic and $d\geq 2$. If $\chi_d^k$ is the character of the $S_d$-representation $H^k(\PConf_d(\CC)/\CC^\times,\QQ)$, then
\[
    E_d^{\sfr}(P) = \sum_{k=0}^{d-2} \frac{(-1)^k \langle P, \chi_d^k \rangle}{q^k}.
\]
\end{thm}

\begin{proof}
The proof is parallel to that of Theorem \ref{thm twisted gl}. We begin with
\[
    E_d^\sfr(P) = \frac{1}{|\sfpoly_d(\FF_q)|}\sum_{f\in \sfpoly_d(\FF_q)} P(f) = \sum_{\lambda \vdash d} P(\lambda) \nu^\sfr(\lambda),
\]
and use Theorem \ref{thm rep interp} to conclude
\begin{align*}
    E_d^\sfr(P) &= \sum_{\lambda \vdash d} P(\lambda) \nu^\sfr(\lambda)\\
    &= \sum_{\lambda \vdash d} \frac{1}{z_\lambda} \sum_{k=0}^{d-2} \frac{(-1)^k P(\lambda)\chi_d^k(\lambda)}{q^k}\\
    &= \sum_{k=0}^{d-2} \frac{(-1)^k}{q^k} \left(\sum_{\lambda \vdash d} \frac{P(\lambda) \chi_d^k(\lambda)}{z_\lambda}\right)\\
    &= \sum_{k=0}^{d-2} \frac{(-1)^k \langle P, \chi_d^k \rangle}{q^k}.
\end{align*}
\end{proof}

The \'{e}tale cohomological approach to Theorem \ref{thm sf twisted gl} taken in \cite{CEF} connects squarefree polynomials over $\FF_q$ with the configuration space of points on the affine line. The geometric perspective seems to break down in the case of Theorem \ref{thm twisted gl}: There is no apparent correspondence between configurations of distinct points in $\RR^3$ and monic polynomials over $\FF_q$. We would be interested in a cogent geometric explanation for the relationship between the representations $H^{2k}(\PConf_d(\RR^3),\QQ)$ and the expected value of factorization statistics on $\poly_d(\FF_q)$.

\subsection{Asymptotic stability}
\label{sec asymp stab}
Church \cite[Thm. 1]{Church} showed that for all $k\geq 0$ and $n \geq 2$ the families of symmetric group representations $H^k(\PConf_d(\RR^n),\QQ)$ are \emph{representation stable}. We do not require the details of representation stability (the interested reader should consult \cite{Church Farb},) only the following fact \cite[Sec. 3.4]{CEF} which we take as a black box: If $P$ is a factorization statistic given by a character polynomial (see Example \ref{ex 1} (3)) and $A_d$ is a sequence of $S_d$-representations with characters $\alpha_d$ which exhibit ``representation stability,'' then the sequence of inner products $\langle P, \alpha_d \rangle$ is eventually constant. In that case we write $\langle P, \alpha\rangle$ for the limit of $\langle P, \alpha_d \rangle$ as $d\rightarrow\infty$.

Church, Ellenberg, and Farb use the representation stability of $H^k(\PConf_d(\CC),\QQ)$ to prove Theorem \ref{thm CEF asymp}.

\begin{thm}[{\cite[Thm. 1]{CEF} }]
\label{thm CEF asymp}
Let $P$ be a factorization statistic given by a character polynomial and write $\langle P, \phi^k\rangle$ for the limit of $\langle P, \phi_d^k \rangle$ as $d\rightarrow\infty$. Then
\[
    \lim_{d\rightarrow\infty}\frac{1}{q^d}\sum_{f\in \poly_d^\sfr(\FF_q)}P(f) = \sum_{k = 0}^\infty \frac{(-1)^k \langle P, \phi^k\rangle}{q^k}.
\]
\end{thm}

Church's theorem implies that for each $k$, $H^{2k}(\PConf_d(\RR^3),\QQ)$ is representation stable. Hyde and Lagarias showed that $H^k(\PConf_d(\CC)/\CC^\times,\QQ)\cong \beta_{[k]}(\Pi_d)$ as $S_d$-representations where $\beta_{[k]}(\Pi_d)$ are the \emph{rank-selected homology of the partition lattice}. Hersh and Reiner \cite[Thm. 1.8]{Hersh Reiner} showed that $\beta_{[k]}(\Pi_d)$ is representation stable. Therefore we deduce the asymptotic stability of expected values from Theorems \ref{thm twisted gl} and \ref{thm sf twisted gl}.

\begin{thm}[Asymptotic stability for expected values]
\label{thm asymp stab}
Let $P$ be a factorization statistic given by a character polynomial (see Section \ref{ex 1} (3).) Then
\[
    \lim_{d\rightarrow\infty}E_d(P) = \sum_{k=0}^\infty \frac{\langle P, \psi^k\rangle}{q^k} \hspace{3em}
    \lim_{d\rightarrow\infty}E_d^\sfr(P) = \sum_{k=0}^\infty \frac{(-1)^k\langle P, \chi^k\rangle}{q^k},
\]
where the limits are taken $1/q$-adically.
\end{thm}

\subsection{Constraint on $E_d(P)$ coefficients}

Theorem \ref{thm constraint} below identifies the total cohomology of $\PConf_d(\RR^3)$ with the regular representation $\QQ[S_d]$.

\begin{thm}
\label{thm constraint}
For each $d\geq 1$ there is an isomorphism of $S_d$-representations
\begin{equation}
\label{eqn cohom}
    \bigoplus_{k=0}^{d-1}H^{2k}(\PConf_d(\RR^3),\QQ) \cong \QQ[S_d],
\end{equation}
where $\QQ[S_d]$ is the regular representation of $S_d$.
\end{thm}

\begin{proof}
Let $\rho$ be the character of $\bigoplus_{k=0}^{d-1}H^{2k}(\PConf_d(\RR^3),\QQ)$. Then
\[
    \rho = \sum_{k=0}^{d-1} \psi_d^k,
\]
where $\psi_d^k$ is the character of $H^{2k}(\PConf_d(\RR^3),\QQ)$. It suffices to show that $\rho$ is equal to the character of the regular representation, that is
\[
    \rho(\lambda) = \begin{cases} d! & \lambda = [1^d]\\ 0 & \text{otherwise.} \end{cases}
\]
By Theorem \ref{thm rep interp} we have
\[
    \nu(\lambda) = \frac{1}{z_\lambda}\sum_{k=0}^{d-1} \frac{\psi_d^k(\lambda)}{q^k},
\]
where $\nu$ is the splitting measure defined by
\[
    \nu(\lambda) = \frac{1}{q^d}\prod_{j\geq 1}\multi{M_j(q)}{m_j}.
\]
Let $\nu_1$ denote the splitting measure evaluated at $q = 1$. Then $\nu_1(\lambda) = \frac{\rho(\lambda)}{z_\lambda}$. On the other hand, $M_j(1) = 0$ for $j > 1$ and $M_1(1) = 1$ so
\[
    \nu_1(\lambda) = \prod_{j\geq 1}\multi{M_j(1)}{m_j} = \begin{cases} 1 & \lambda = [1^d]\\ 0 & \text{otherwise.} \end{cases}
\]
Since $z_{[1^d]} = d!$ the result follows.
\end{proof}

We have been told that, from the point of view of Lie theory, that  Theorem \ref{thm constraint} is a consequence of the Poincar\'{e}-Brikhoff-Witt theorem---this is outside our expertise. The proof of Theorem \ref{thm constraint} shows that \eqref{eqn cohom} also follows naturally from the splitting measure perspective which is how we came into it. The following corollary will be used in Section \ref{section example} to explain common phenomena that arise in expected value computations for factorization statistics.

\begin{cor}
\label{thm coeffs constraint}
Suppose $P$ is a factorization statistic defined on $\poly_d(\FF_q)$ which, viewed as a class function of $S_d$, is the character of an $S_d$-representation $V$. Let $E_d(P)$ be the expected value of $P$ on $\poly_d(\FF_q)$.
\begin{enumerate}
    \item $E_d(P)$ is a polynomial in $1/q$ of degree at most $d-1$ with non-negative integer coefficients.
    
    \item The evaluation of $E_d(P)$ at $q = 1$ is $E_d(P)_{q = 1} = \dim V$.
\end{enumerate}
\end{cor}

\begin{proof}
(1) Recall that the inner product $\langle \chi, \psi \rangle$ of characters is the dimension of the vector space of maps between the corresponding representations, hence is a non-negative integer. Thus if $P$ is an $S_d$-character then Theorem \ref{thm twisted gl} implies that
\[
    E_d(P) = \sum_{k=0}^{d-1} \frac{\langle P, \psi_d^k \rangle}{q^k},
\]
has non-negative coefficients.

(2) The inner product of class functions is bilinear. Therefore, by Theorem \ref{thm constraint}
\[
    E_d(P)_{q = 1} = \sum_{k=0}^{d-1} \langle P, \psi_d^k\rangle = \langle P, \sum_{k=0}^{d-1} \psi_d^k \rangle = \langle P, \chi_{\mathrm{reg}} \rangle.
\]
It follows from the general representation theory of finite groups that $\langle P, \chi_{\mathrm{reg}} \rangle = \dim V$. Therefore,
\[
    E_d(P)_{q = 1} = \dim V.
\]
\end{proof}

\section{Examples}
\label{section example}

The twisted Grothendieck-Lefschetz formulas (Theorems \ref{thm sf twisted gl} and \ref{thm twisted gl}) form a bridge connecting polynomial factorization statistics on the one hand and representations of the symmetric group and cohomology of configuration spaces on the other. Translating information back and forth across this bridge leads to an interesting interplay among these structures. In this section we first revisit the example of quadratic excess $Q$ to see how our results explain the properties of $E_d(Q)$ observed in the introduction. We finish with some results on expected values and the structure of $H^{2k}(\PConf_d(\RR^3),\QQ)$ using the constraint provided by Theorem \ref{thm constraint}.

\subsection{Quadratic excess}
Recall the quadratic excess factorization statistic $Q$ from the introduction: $Q(f)$ is defined as the difference between the number of reducible versus irreducible quadratic factors of $f$. Rephrasing this in terms of partitions, if $x_k(\lambda)$ is the number of parts of $\lambda$ of size $k$, then
\[
    Q(\lambda) = \binom{x_1(\lambda)}{2} - \binom{x_2(\lambda)}{1}.
\]
Let $\QQ[d]$ be the permutation representation of the symmetric group with basis $\{e_1, e_2, \ldots , e_d\}$ and consider the representation given by the second exterior power $\bigwedge^2 \QQ[d]$. This representation has dimension $\binom{d}{2}$ with a natural basis given by 
\[
    \{e_i \wedge e_j : i < j\}.
\]
If $\sigma \in S_d$ is a permutation, then the trace of $\sigma$ on $\bigwedge^2 \QQ[d]$ is
\begin{align*}
    \mathrm{Trace}(\sigma) &= \#\{\{i,j\} : \sigma \text{ fixes $i$ and $j$}\} - \#\{\{i,j\} : \sigma \text{ transposes $i$ and $j$}\}\\
    &= \binom{x_1(\sigma)}{2} - \binom{x_2(\sigma)}{1}\\
    &= Q(\sigma).
\end{align*}
Thus $Q$, viewed as a class function of $S_d$, is the character of $\bigwedge^2\QQ[d]$. It follows from Corollary \ref{thm coeffs constraint} that coefficients of $E_d(Q)$ are non-negative integers summing to $\binom{d}{2} = \dim \bigwedge^2 \QQ[d]$. The coefficientwise convergence of $E_d(Q)$ follows from Theorem \ref{thm asymp stab}. The $1/q$-adic limit of $E_d(Q)$ as $d\rightarrow\infty$ is a rational function of $q$, which explains the simple pattern emerging in the coefficients of $E_d(q)$. In particular, using \cite[Cor. 10]{Chen 2} we compute,
\begin{align*}
    \lim_{d\rightarrow\infty}E_d(Q) &= \frac{1}{2}\bigg(1 + \frac{1}{q}\bigg)\bigg(\frac{1}{1 - \frac{1}{q}}\bigg)^2 - \frac{1}{2}\bigg(1 - \frac{1}{q}\bigg)\bigg(\frac{1}{1 - \frac{1}{q^2}}\bigg)\\
    &= \tfrac{2}{q} + \tfrac{2}{q^2} + \tfrac{4}{q^3} + \tfrac{4}{q^4} + \tfrac{6}{q^5} + \tfrac{6}{q^6} + \tfrac{8}{q^7} + \tfrac{8}{q^8} + \tfrac{10}{q^9} + \ldots
\end{align*}

\subsection{Identifying irreducible components}
Theorem \ref{thm constraint} gives a constraint on the cohomology of $\PConf_d(\RR^3)$,
\[
    \bigoplus_{k=0}^{d-1}H^{2k}(\PConf_d(\RR^3),\QQ) \cong \QQ[S_d],
\]
where $\QQ[S_d]$ is the regular representation of the symmetric group. The regular representation of $S_d$ is well-understood: the irreducible representations of $S_d$ are indexed by partitions $\lambda\vdash d$, each irreducible $\Sp_\lambda$ is a direct summand of $\QQ[S_d]$ with multiplicity $f_\lambda := \dim \Sp_\lambda$. Thus Theorem \ref{thm constraint} tells us that the irreducible components $\Sp_\lambda$ of $\QQ[S_d]$ are distributed among the various degrees of cohomology on the left hand side of \eqref{eqn cohom}. The twisted Grothendieck-Lefschetz formula for $\poly_d(\FF_q)$ implies that the filtration of the regular representation given by Theorem \ref{thm constraint} completely determines and is determined by the expected values of factorization statistics on $\poly_d(\FF_q)$. We use Theorem \ref{thm constraint} to identify the degrees of some of the irreducible $S_d$-representations in the cohomology of $\PConf_d(\RR^3)$.

\subsection{Trivial representation} Let $\tr = \Sp_{[d]}$ be the one-dimensional \emph{trivial representation} of $S_d$. The character of the trivial representation is constant equal to 1. Interpreting the trivial character as a factorization statistic we have $E_d(1) = 1$ and Theorem \ref{thm twisted gl} implies
\[
    1 = E_d(1) = \sum_{k=0}^{d-1}\frac{\langle 1, \psi_d^k\rangle}{q^k}.
\]
Comparing coefficients of $1/q^k$ we conclude that $\langle 1, \psi_d^0\rangle = 1$ and $\langle 1, \psi_d^k\rangle = 0$ for $k > 0$. Hence, $\tr$ is a summand of $H^0(\PConf_d(\RR^3),\QQ)$. On the other hand, $\PConf_d(\RR^3)$ is path connected so $ H^0(\PConf_d(\RR^3),\QQ)$ is one-dimensional. Thus
\begin{equation}
\label{eqn trivial}
    H^0(\PConf_d(\RR^3),\QQ) \cong \tr,
\end{equation}
and $H^{2k}(\PConf_d(\RR^3),\QQ)$ has no trivial component for $k > 0$.

Recall that the characters $\chi_\lambda$ of the irreducible representations $\Sp_\lambda$ of $S_d$ form a $\QQ$-basis for the vector space of all class functions. If $P$ is a factorization statistic, then there are $a_\lambda(P)\in \QQ$ such that
\[
    P = \sum_{\lambda\vdash d} a_\lambda(P) \chi_\lambda,
\]
where $\chi_\lambda$ is the character of the irreducible representation $\Sp_\lambda$. In particular if $a_1(P) := a_{[d]}(P)$ is the coefficient of the trivial character in this decomposition, then we have the following corollary.

\begin{cor}
If $P$ is any factorization statistic and $a_1(P)$ is the coefficient of the trivial character in the expression of $P$ as a linear combination of irreducible $S_d$-characters, then
\[
    a_1(P) = \lim_{q\rightarrow\infty} E_d(P).
\]
Hence $a_1(P) = 0$ if and only if the expected value of $P$ approaches 0 for large $q$.
\end{cor}

\subsection{Sign representation} Let $\Sgn := \Sp_{[1^d]}$ be the one-dimensional \emph{sign representation}. The character of $\Sgn$ is $\sgn(\lambda) = (-1)^{d - \ell(\lambda)}$, or equivalently $\sgn([j]) = (-1)^{j - 1}$ for a partition $[j]$ with one part of size $j$ and then $\sgn$ extends multiplicatively to partitions with more than one part. Viewing $\sgn$ as a factorization statistic Theorem \ref{thm twisted gl} implies
\[
    E_d(\sgn) = \sum_{k=0}^{d-1} \frac{\langle \sgn, \psi_d^k\rangle}{q^k}.
\]
On the other hand, Corollary \ref{thm coeffs constraint} tells us that $\langle \sgn, \psi_d^k \rangle = 1$ for exactly one $k$ and is 0 otherwise---which value of $k$ is it?

\begin{thm}[\cite{Hyde}]
\label{thm sgn}
For each $d\geq 1$,
\[
    E_d(\sgn) = \frac{1}{q^{\lfloor d/2 \rfloor}}.
\]
Hence $H^{2\lfloor d/2 \rfloor}(\PConf_d(\RR^3),\QQ)$ is the unique cohomological degree with a $\Sgn$ summand.
\end{thm}

We prove Theorem \ref{thm sgn} in \cite{Hyde} using our \emph{liminal reciprocity theorem} which relates factorization statistics in $\poly_d(\FF_q)$ with the limiting values of \emph{squarefree} factorization statistics for $\FF_q[x_1, x_2, \ldots, x_n]$ as the number of variables $n$ tends to infinity.

Theorem \ref{thm sgn} has a surprising consequence. The \emph{even type} factorization statistic $ET$ is defined by $ET(f) = 1$ when the factorization type of $f$ is an even partition and $ET(f) = 0$ otherwise. Thus the expected value $E_d(ET)$ is the probability of a random polynomial in $\poly_d(\FF_q)$ having even factorization type. One might guess that a polynomial should be just as likely to have an even versus odd factorization type. However, notice that
\[
    ET = \tfrac{1}{2}(1 + \sgn)
\]
as class functions of $S_d$. It follows by the linearity of expectation that
\[
    E_d(ET) = \tfrac{1}{2}(E_d(1) + E_d(\sgn)) = \tfrac{1}{2}\big(1 + \tfrac{1}{q^{\lfloor d/2 \rfloor}}\big).
\]
The leading term of this probability is $1/2$ as we expected, but there is a slight bias toward a polynomial having even factorization type coming from the sign representation and the degree of cohomology in which it appears. For comparison we remark that in the squarefree case the probability of a random polynomial in $\poly_d^\sfr(\FF_q)$ having even factorization type is exactly
\[
    E_d^\sfr(ET) = \tfrac{1}{2},
\]
matching our original guess.

\subsection{Standard representation}
Let $\QQ[d]$ be the permutation representation of $S_d$. The irreducible decomposition of $\QQ[d]$ is
\[
    \QQ[d] \cong \tr \oplus \Std,
\]
where $\Std = \Sp_{[d-1,1]}$ is the $(d-1)$-dimensional \emph{standard representation} of $S_d$. Let $R$ be the character of $\QQ[d]$. If $\sigma\in S_d$, then $R(\sigma)$ is the number of fixed points of $\sigma$ acting on the set $\{1, 2, \ldots, d\}$; hence $R(\lambda) = x_1(\lambda)$ is the number of parts of $\lambda$ of size one. Viewed as a factorization statistic, $R(f)$ counts the number of $\FF_q$-roots of $f$ with multiplicity.

\begin{thm}
Let $R(f)$ be the number of $\FF_q$-roots with multiplicity of $f\in \poly_d(\FF_q)$. Then the expected value $E_d(R)$ of $R$ on $\poly_d(\FF_q)$ is
\begin{equation}
\label{eqn roots}
    E_d(R) = \frac{1 - \frac{1}{q^d}}{1 - \frac{1}{q}} = 1 + \frac{1}{q} + \frac{1}{q^2} + \frac{1}{q^3} + \ldots + \frac{1}{q^{d-1}}.
\end{equation}
It follows that the multiplicity of $\Std$ in $H^{2k}(\PConf_d(\RR^3),\QQ)$ is 1 for $0 < k < d$.
\end{thm}

\begin{proof}
First note that
\[
    E_d(R) = \frac{1}{q^d}\sum_{f\in \poly_d(\FF_q)}R(f) = \sum_{\lambda \vdash d} x_1(\lambda)\nu(\lambda),
\]
where $\nu$ is the splitting measure. In the course of proving Theorem \ref{thm rep interp} we derived the following formal power series identity,
\begin{equation}
\label{eqn formal identity}
    \sum_{d\geq 0}\sum_{\lambda \vdash d}\nu(\lambda)a^\lambda = \prod_{j\geq 1}\left(\frac{1}{1 - a_j/q^j}\right)^{M_j(q)}.
\end{equation}

Consider the effect of the operator $a_1 \frac{\partial}{\partial a_1}$ on \eqref{eqn formal identity}. On the left hand side we get
\[
    a_1 \frac{\partial}{\partial a_1}\sum_{d\geq 0}\sum_{\lambda \vdash d}\nu(\lambda)a^\lambda = \sum_{d\geq 1}\sum_{\lambda \vdash d}x_1(\lambda)\nu(\lambda)a^\lambda.
\]
On the right hand side we have
\[
    a_1 \frac{\partial}{\partial a_1}\prod_{j\geq 1}\left(\frac{1}{1 - a_j/q^j}\right)^{M_j(q)} = \frac{M_1(q)a_1}{q(1 - a_1/q)}\prod_{j\geq 1}\left(\frac{1}{1 - a_j/q^j}\right)^{M_j(q)}.
\]
Now substitute $a_j \mapsto t^j$ for all $j$ to arrive at
\[
    \sum_{d\geq 1}\sum_{\lambda \vdash d}x_1(\lambda)\nu(\lambda)t^d = \sum_{d\geq 1}E_d(R)t^d,
\]
on the left and
\begin{align*}
    \frac{M_1(q)t}{q(1 - t/q)}\prod_{j\geq 1}\left(\frac{1}{1 - t^j/q^j}\right)^{M_j(q)} &= \frac{t}{1 - t/q}\prod_{j\geq 1}\left(\frac{1}{1 - (t/q)^j}\right)^{M_j(q)}\\
    &= \frac{t}{1 - t/q} \cdot \frac{1}{1 - t}
\end{align*}
on the right, where the last equality is a consequence of the \emph{cyclotomic identity}:
\[
    \frac{1}{1 - qt} = \prod_{j\geq 1}\left(\frac{1}{1 - t^j}\right)^{M_j(q)}.
\]
Together this becomes
\begin{equation}
\label{eqn specialization}
    \sum_{d\geq 1}E_d(R)t^d = \frac{t}{1 - t/q} \cdot \frac{1}{1 - t}.
\end{equation}
Expanding the right hand side of \eqref{eqn specialization} gives
\[
    \frac{t}{1 - t/q} \cdot \frac{1}{1 - t} = \frac{1}{1 - t}\sum_{d\geq 1} \frac{1}{q^{d-1}}t^d = \sum_{d\geq 1}\left(\frac{1 - \frac{1}{q^d}}{1 - \frac{1}{q}}\right)t^d.
\]
Comparing coefficients of $t^d$ we conclude that
\[
    E_d(R) = \frac{1 - \frac{1}{q^d}}{1 - \frac{1}{q}} = 1 + \frac{1}{q} + \frac{1}{q^2} + \frac{1}{q^3} + \ldots + \frac{1}{q^{d-1}}.
\]
The assertions about the multiplicity of $\Std$ in $H^{2k}(\PConf_d(\RR^3),\QQ)$ follow from Theorem \ref{thm twisted gl} and \eqref{eqn trivial}.
\end{proof}


\begin{thebibliography}{99}


\bibitem{Calderbank Hanlon Robinson}
A. R. Calderbank, P. Hanlon, and R. W. Robinson, Partitions into even and odd block size and some unusual characters of the symmetric groups. \emph{Proc. London Math. Soc.} \textbf{3} (1986), no. 2, 288-320.

\bibitem{Chen 1}
W. Chen, "Analytic number theory for 0-cycles." \emph{Math. Proc. Cambridge Philos. Soc.}, Cambridge University Press, (2017), 1-24.

\bibitem{Chen 2}
W. Chen, Twisted cohomology of configuration spaces and spaces of maximal tori via point-counting, arXiv:1603.03931.

\bibitem{Church}
T. Church, Homological stability for configuration spaces of manifolds. \emph{Invent. Math.} \textbf{188} (2012) 465-504.

\bibitem{Church Farb}
T. Church, B. Farb, Representation theory and homological stability. \emph{Adv. Math.} \textbf{245} (2013) 250-314.

\bibitem{CEF}
T. Church, J. Ellenberg, and B. Farb. Representation stability in cohomology and asymptotics
for families of varieties over finite fields. \emph{Contemp. Math.} \textbf{620} (2014),
1-54.

\bibitem{Fulman}
J. Fulman, A generating function approach to counting theorems for square-free polynomials and maximal tori, \emph{Ann. Comb.} \textbf{20} (2016), 587-599.

\bibitem{Gadish}
N. Gadish, A trace formula for the distribution of rational $G$-orbits in
ramified covers, adapted to representation stability, \emph{to appear in NY J. of Math}.

\bibitem{Hanlon 1}
P. Hanlon, The action of $S_n$ on the components of the Hodge decomposition of Hochschild homology.
\emph{Michigan Math. J.} \textbf{37} (1990), 105-124.

\bibitem{Hast Matei}
D. Hast, V. Matei, Higher moments of arithmetic functions in short intervals: a geometric perspective, arXiv:1604.02067 (2016).

\bibitem{Hersh Reiner}
P. Hersh, V. Reiner. Representation Stability for Cohomology of Configuration Spaces in $\RR^d$. \emph{Internat. Math. Res. Notices} \textbf{2017} (2016), no. 5, 1433-1486.

\bibitem{Hyde}
T. Hyde, Liminal factorization statistics and reciprocity, \emph{In preparation.}

\bibitem{Hyde Lagarias}
T. Hyde, J. C. Lagarias, Polynomial splitting measures and cohomology of the pure braid group. \emph{Arnold Math. J.} (2017), 1-31.

\bibitem{Rosen}
M. Rosen, \emph{Number theory in function fields}. \textbf{120} Springer Science \& Business Media,  (2013).

\bibitem{Sundaram Welker}
S. Sundaram and V. Welker, Group actions on arrangements of linear subspaces and applications to
configuration spaces. \emph{Trans. Amer. Math. Soc.} \textbf{349} (1997), no. 4, 1389-1420.
\end{thebibliography}
\end{document}